\documentclass{svmult}

\usepackage{amsfonts}
\usepackage[mathscr]{eucal}
\usepackage{amsmath}
\usepackage{amssymb}

\usepackage[bottom]{footmisc}

\newcommand\bC{{\mathbb C}}

\newcommand\bG{{\mathbb G}}

\newcommand\bP{{\mathbb P}}

\newcommand\bR{{\mathbb R}}

\newcommand\cC{{\mathcal C}}

\newcommand\cE{{\mathcal E}}

\newcommand\cI{{\mathcal I}}
\newcommand\cL{{\mathcal L}}
\newcommand\cM{{\mathcal M}}
\newcommand\cN{{\mathcal N}}
\newcommand\cO{{\mathcal O}}

\newcommand\fg{\mathfrak{g}}
\newcommand\fl{\mathfrak{l}}

\newcommand\charc{{\rm char}}

\newcommand\diag{{\rm diag}}

\newcommand\id{{\rm id}}

\newcommand\op{{\rm op}}

\newcommand\red{{\rm red}}

\newcommand\Aff{{\rm Aff}}

\newcommand\Aut{{\rm Aut}}

\newcommand\Der{{\rm Der}}
\newcommand\End{{\rm End}}
\newcommand\Ext{{\rm Ext}}

\newcommand\GL{{\rm GL}}
\newcommand\Hilb{{\rm Hilb}}
\newcommand\Hom{{\rm Hom}}

\newcommand\Kern{{\rm Ker}}
\newcommand\Lie{{\rm Lie}}

\newcommand\PGL{{\rm PGL}}

\newcommand\Sec{{\rm Sec}}
\newcommand\Spec{{\rm Spec}}
\newcommand\Spf{{\rm Spf}}

\begin{document}

\title*{On automorphisms and endomorphisms of projective varieties}

\author{Michel Brion}

\institute{Michel Brion 
\at Institut Fourier, Universit\'e de Grenoble\\
B.P.~74, 38402 Saint-Martin d'H\`eres Cedex, France\\ 
\email{Michel.Brion@ujf-grenoble.fr}}

\maketitle


 
\begin{abstract}
\, We first show that any connected algebraic group over a perfect 
field is the neutral component of the automorphism group scheme 
of some normal projective variety. Then we show that very few connected 
algebraic semigroups can be realized as endomorphisms of some 
projective variety $X$, by describing the structure of all connected 
subsemigroup schemes of End($X$).
\end{abstract}

\section{Introduction and statement of the results}
\label{sec:int}

By a result of Winkelmann (see \cite{Winkelmann}), every connected real 
Lie group $G$ can be realized as the automorphism group of some complex 
Stein manifold $X$, which may be chosen complete, and hyperbolic in the
sense of Kobayashi. Subsequently, Kan showed in \cite{Kan} that we may 
further assume $\dim_{\bC}(X) = \dim_{\bR}(G)$. 

We shall obtain a somewhat similar result for connected algebraic 
groups. We first introduce some notation and conventions, and 
recall general results on automorphism group schemes.

Throughout this article, we consider schemes and their morphisms 
over a fixed field $k$. Schemes are assumed to be separated; 
subschemes are locally closed unless mentioned otherwise. 
By a \emph{point} of a scheme $S$, we mean a $T$-valued point 
$f : T \to S$ for some scheme $T$. A \emph{variety} is 
a geometrically integral scheme of finite type.

We shall use \cite{SGA3} as a general reference for group schemes.
We denote by $e_G$ the neutral element of a group scheme $G$, and 
by $G^o$ the neutral component. An \emph{algebraic group} is 
a smooth group scheme of finite type.

Given a projective scheme $X$, the functor of automorphisms,
\[ T \longmapsto \Aut_T(X \times T), \] 
is represented by a group scheme, locally of finite type, 
that we denote by $\Aut(X)$. The Lie algebra of $\Aut(X)$ 
is identified with the Lie algebra of global vector fields,
$\Der(\cO_X)$ (these results hold more generally for projective
schemes over an arbitrary base, see \cite[p.~268]{Grothendieck};
they also hold for proper schemes of finite type over a field, 
see \cite[Thm.~3.7]{Matsumura-Oort}). In particular, the 
neutral component, $\Aut^o(X)$, is a group scheme of finite type;
when $k$ is perfect, the reduced subscheme, $\Aut^o(X)_{\red}$, 
is a connected algebraic group. As a consequence, $\Aut^o(X)$
is a connected algebraic group if $\charc(k) = 0$, since every 
group scheme of finite type is reduced under that assumption.
Yet $\Aut^o(X)$ is not necessarily reduced in prime 
characteristics (see e.g.~the examples in 
\cite[\S 4]{Matsumura-Oort}).

We may now state our first result:

\begin{theorem}\label{thm:aut}
Let $G$ be a connected algebraic group, and $n$ its dimension.

\smallskip

\noindent
If $\charc(k) = 0$, then there exists a smooth projective 
variety $X$ such that $\Aut^o(X) \cong G$ and $\dim(X) = 2 n$.

\smallskip

\noindent
If $\charc(k) > 0$ and $k$ is perfect, then there exists 
a normal projective variety $X$ such that 
$\Aut^0_{\red}(X) \cong G$ and $\dim(X) = 2 n$ 
(resp. $\Aut^o(X) \cong G$ and $\dim(X) = 2 n + 2)$.
\end{theorem}

This result is proved in Section \ref{sec:aut}, first in the case 
where $\charc(k) = 0$; then we adapt the arguments to the case 
of prime characteristics, which is technically more involved 
due to group schemes issues. We rely on fundamental results about 
the structure and actions of algebraic groups over an algebraically 
closed field, for which we refer to the recent exposition \cite{BSU}.

Theorem \ref{thm:aut} leaves open many basic questions about
automorphism group schemes. For instance, can one realize 
every connected algebraic group over an arbitrary field 
(or more generally, every connected group scheme of finite type) 
as the full automorphism group scheme of a normal projective variety? 
Also, very little seems to be known about the group of components, 
$\Aut(X)/\Aut^o(X)$, where $X$ is a projective variety. In 
particular, the question of the finite generation of this group 
is open, already when $X$ is a complex projective manifold.  

As a consequence of Theorem \ref{thm:aut}, we obtain
the following characterization of Lie algebras of vector fields:

\begin{corollary}\label{cor:der}
Let $\fg$ be a finite-dimensional Lie algebra over a field $k$ 
of characteristic $0$. Then the following conditions are equivalent:

\begin{enumerate}

\item[{\rm (i)}]{$\fg \cong \Der(\cO_X)$ for some proper scheme 
$X$ of finite type.}

\item[{\rm (ii)}]{ $\fg$ is the Lie algebra of a linear algebraic
group.}

\end{enumerate}

Under either condition, $X$ may be chosen projective, smooth, 
and unirational of dimension $2n$, where $n := \dim(\fg)$.
If $k$ is algebraically closed, then we may further choose $X$
rational. 
\end{corollary}  

This result is proved in Subsection \ref{subsec:der}. 
The Lie algebras of linear algebraic groups over a field of 
characteristic $0$ are called algebraic Lie algebras; they have been 
characterized by Chevalley in \cite{Chevalley51, Chevalley55}. 
More specifically, a finite-dimensional Lie algebra $\fg$ is algebraic 
if and only if its image under the adjoint representation is an 
algebraic Lie subalgebra of $\fg\fl(\fg)$ (see 
\cite[Chap.~V, \S 5, Prop.~3] {Chevalley55}). Moreover, 
the algebraic Lie subalgebras of $\fg \fl (V)$, where $V$ is 
a finite-dimensional vector space, are characterized in 
\cite[Chap.~II, \S 14]{Chevalley51}. Also, recall a result of 
Hochschild (see \cite{Hochschild}): the isomorphism classes of 
algebraic Lie algebras are in bijective correspondence with the 
isomorphism classes of connected linear algebraic groups with 
unipotent center.

In characteristic $p > 0$, one should rather consider restricted 
Lie algebras, also known as $p$-Lie algebras. In this setting, 
characterizing Lie algebras of vector fields seems to be an open 
question. This is related to the question of characterizing
automorphism group schemes, via the identification of restricted 
Lie algebras with infinitesimal group schemes of height $\leq 1$ 
(see \cite[Exp.~VIIA, Thm.~7.4]{SGA3}).

\medskip

Next, we turn to the monoid schemes of endomorphisms of projective 
varieties; we shall describe their connected subsemigroup schemes. 
For this, we recall basic results on schemes of morphisms. 

Given two projective schemes $X$ and $Y$, the functor of morphisms,
\[ T \longmapsto \Hom_T(X \times T, Y \times T) 
\cong \Hom(X \times T,Y), \]
is represented by an open subscheme of the Hilbert scheme
$\Hilb(X \times Y)$, by assigning to each morphism its graph 
(see \cite[p.~268]{Grothendieck}, and \cite[\S 1.10]{Kollar96},
\cite[\S 4.6.6]{Sernesi} for more details). 
We denote that open subscheme by $\Hom(X,Y)$. The composition 
of morphisms yields a natural transformation of functors, and 
hence a morphism of schemes
\[ \Hom(X,Y) \times \Hom(Y,Z) \longrightarrow \Hom(X,Z), 
\quad (f,g) \longmapsto g f \] 
where $Z$ is another projective scheme.

As a consequence of these results, the functor of endomorphisms 
of a projective scheme $X$ is represented by a scheme, $\End(X)$; 
moreover, the composition of endomorphisms equips $\End(X)$ 
with a structure of monoid scheme with neutral element 
being of course the identity, $\id_X$. 
Each connected component of $\End(X)$ is of finite type, 
and these components form a countable set. 
The automorphism group scheme $\Aut(X)$ is open in $\End(X)$ 
by \cite[p.~267]{Grothendieck} (see also
\cite[Lem.~I.1.10.1]{Kollar96}). If $X$ is a variety, 
then $\Aut(X)$ is also closed in $\End(X)$, as follows from 
\cite[Lem.~4.4.4]{Brion12}; thus, $\Aut(X)$ is a union of 
connected components of $\End(X)$. In particular, $\Aut^o(X)$ 
is the connected component of $\id_X$ in $\End(X)$.

As another consequence, given a morphism $f : X \to Y$ 
of projective schemes, the functor of sections of $f$ 
is represented by a scheme that we shall denote by $\Sec(f)$: 
the fiber at $\id_Y$ of the morphism
\[ \lambda_f : \Hom(Y,X) \longrightarrow \End(Y), 
\quad g \longmapsto f g. \]
Every section of $f$ is a closed immersion;
moreover, $\Sec(f)$ is identified with an open subscheme 
of $\Hilb(X)$ by assigning to each section its image 
(see \cite[p.~268]{Grothendieck} again; our notation differs 
from the one used there). Given a section $s \in \Sec(f)(k)$,
we may identify $Y$ with the closed subscheme $Z := s(Y)$;
then $f$ is identified with a \emph{retraction} of $X$ 
onto that subscheme, i.e., to a morphism $r : X \to Z$ such that
$r i = \id_Z$, where $i : Z \to X$ denotes the inclusion.
Moreover, the endomorphism $e := ir$ of $X$ is
\emph{idempotent}, i.e., satisfies $e^2 = e$.

Conversely, every idempotent $k$-rational point of $\End(X)$
can be written uniquely as $e = i r$, where $i : Y \to X$
is the inclusion of the image of $e$ (which coincides
with its fixed point subscheme), and $r: X \to Y$
is a retraction. When $X$ is a variety, $Y$ is
a projective variety as well. We now analyze the connected component
of $e$ in $\End(X)$:

\begin{proposition}\label{prop:end}
Let $X$ be a projective variety, $e \in \End(X)(k)$ an idempotent,
and $C$ the connected component of $e$ in $\End(X)$.
Write $e = i r$, where $i : Y \to X$ denotes the inclusion of 
a closed subvariety, and $r: X \to Y$ is a retraction.

\begin{enumerate}

\item[{\rm (i)}]{The morphism
\[ \rho_r : \Hom(Y,X) \longrightarrow \End(X),
\quad f \longmapsto f r \]
restricts to an isomorphism from the connected component of
$i$ in $\Hom(Y,X)$, to $C$. Moreover, $C$ is a subsemigroup scheme 
of $\End(X)$, and $f = f e$ for any $f \in C$.}

\item[{\rm (ii)}]{The morphism 
\[ \lambda_i \rho_r : \End(Y) \longrightarrow \End(X),
\quad f \longmapsto i f r \]
restricts to an isomorphism of semigroup schemes
$\Aut^o(Y) \stackrel{\cong}{\longrightarrow} eC$.
In particular, $eC$ is a group scheme with neutral
element $e$.} 

\item[{\rm (iii)}]{$\rho_r$ restricts to an isomorphism 
from the connected component of $i$ in $\Sec(r)$, 
to the subscheme $E(C)$ of idempotents in $C$. Moreover, 
$f_1 f_2 = f_1$ for all $f_1,f_2 \in E(C)$; in particular,
$E(C)$ is a closed subsemigroup scheme of $C$.}

\item[{\rm (iv)}]{The morphism 
\[ \varphi : E(C) \times eC \longrightarrow C, \quad
(f,g) \longmapsto f g \] 
is an isomorphism of semigroup schemes, where the
semigroup law on the left-hand side is given by
$(f_1,g_1) \cdot (f_2,g_2) = (f_1,g_1g_2)$.}
\end{enumerate}

\end{proposition}

This is proved in Subsection \ref{subsec:end}, by using 
a version of the rigidity lemma (see \cite[\S 4.4]{Brion12}).
As a straightforward consequence, the maximal connected 
subgroup schemes of $\End(X)$ are exactly the 
$\lambda_i \rho_r( \Aut^o(Y))$ with the above notation
(this fact is easily be checked directly).

As another consequence of Proposition \ref{prop:end}, 
the endomorphism scheme of a projective variety can have
everywhere nonreduced connected components, 
even in characteristic $0$. 
Consider for example a ruled surface 
\[ r : X = \bP(\cE) \longrightarrow Y,\] 
where $Y$ is an elliptic curve and $\cE$ is a locally free 
sheaf on $Y$ which belongs to a nonsplit exact sequence 
\[ 0 \longrightarrow \cO_Y \longrightarrow \cE 
\longrightarrow \cO_Y \longrightarrow 0 \]
(such a sequence exists in view of the isomorphisms
$\Ext^1(\cO_Y, \cO_Y) \cong H^1(Y,\cO_Y) \cong k$). 
Let $i : Y \to X$ be the section associated with the projection
$\cE \to \cO_Y$. Then the image of $i$ yields 
an isolated point of $\Hilb(X)$ with Zariski tangent space 
of dimension $1$ (see e.g. \cite[Ex.~4.6.7]{Sernesi}). 
Thus, the connected component of $i$ in $\Sec(r)$ 
is a nonreduced fat point. By Proposition \ref{prop:end}
(iv), the connected component of $e := ir$ in $\End(X)$ 
is isomorphic to the product of that fat point 
with $\Aut^o(Y) \cong Y$, and hence is nonreduced everywhere. 
This explains a posteriori why we have to be so fussy 
with semigroup schemes.

A further consequence of Proposition \ref{prop:end}
is the following:

\begin{proposition}\label{prop:semi}
Let $X$ be a projective variety, $S$ a connected 
subsemigroup scheme of $\End(X)$, and $E(S) \subset S$
the closed subscheme of idempotents. Assume that $S$ has 
a $k$-rational point.

\begin{enumerate}

\item[{\rm (i)}]{$E(S)$ is a connected subsemigroup scheme 
of $S$, with semigroup law given by $f_1 f_2 = f_1$. Moreover, 
$E(S)$ has a $k$-rational point.}

\item[{\rm (ii)}]{For any $e \in E(S)(k)$, the closed 
subsemigroup scheme $e S \subset S$ is a group scheme. 
Moreover, the morphism 
\[ \varphi : E(S) \times eS  \longrightarrow S, \quad
(f,g) \longmapsto f g \]
is an isomorphism of semigroup schemes.}

\item[{\rm (iii)}]{Identifying $S$ with $E(S) \times eS$ 
via $\varphi$, the projection $\pi : S \to E(S)$ is the 
unique retraction of semigroup schemes from $S$ to $E(S)$.
In particular, $\pi$ is independent of the choice of the
$k$-rational idempotent $e$.}

\end{enumerate}

\end{proposition}

This structure result is proved in Subsection 
\ref{subsec:semi}; a new ingredient is the fact that 
a subsemigroup scheme of a group scheme of finite type 
is a subgroup scheme (Lemma \ref{lem:unit}). The case 
where $S$ has no $k$-rational point is discussed in Remark 
\ref{rem:nopoint} at the end of Subsection \ref{subsec:semi}.

Proposition \ref{prop:semi} yields strong restrictions on
the structure of connected subsemigroup schemes of $\End(X)$, 
where $X$ is a projective variety. For example, if such 
a subsemigroup scheme is commutative or has a neutral element, 
then it is just a group scheme.

As another application of that proposition, we shall
show that the dynamics of an endomorphism of $X$ which 
belongs to some algebraic subsemigroup is very restricted. 
To formulate our result, we need the following:

\begin{definition}\label{def:end}
Let $X$ be a projective variety, and $f$ a $k$-rational
endomorphism of $X$. 

We say that $f$ is \emph{bounded}, if $f$ belongs to 
a subsemigroup of finite type of $\End(X)$. Equivalently, 
the powers $f^n$, where $n \geq 1$, are all contained in 
a finite union of subvarieties of $\End(X)$.

We say that a $k$-rational point $x \in X$ is \emph{periodic}, 
if $x$ is fixed by some $f^n$. 
\end{definition}

\begin{proposition}\label{prop:bound}
Let $f$ be a bounded endomorphism of a projective variety $X$. 

\begin{enumerate}
\item[{\rm (i)}]{There exists a smallest closed algebraic subgroup 
$G \subset \End(X)$ such that $f^n \in G$ for all $n \gg 0$. 
Moreover, $G$ is commutative.}
\item[{\rm (ii)}]{When $k$ is algebraically closed, $f$ has 
a periodic point if and only if $G$ is linear. If $X$ is normal, 
this is equivalent to the assertion that some positive power 
$f^n$ acts on the Albanese variety of $X$ via an idempotent.}
\end{enumerate}

\end{proposition}

This result is proved in Subsection \ref{subsec:bound}. 
As a direct consequence, every bounded endomorphism of 
a normal projective variety $X$ has a periodic point, 
whenever the Albanese variety of $X$ is trivial (e.g., 
when $X$ is unirational); we do not know if any 
such endomorphism has a fixed point. In characteristic $0$,
it is known that every endomorphism (not necessarily bounded) 
of a smooth projective unirational variety $X$ has a fixed point:
this follows from the Woods Hole formula (see 
\cite[Thm.~2]{Atiyah-Bott}, \cite[Exp.~III, Cor.~6.12]{SGA5})
in view of the vanishing of $H^i(X,\cO_X)$ for all $i \geq 1$,
proved e.g. in \cite[Lem.~1]{Serre}.

Also, it would be interesting to extend the above results
to endomorphism schemes of complete varieties. In this setting,
the rigidity lemma of \cite[\S 4.4]{Brion12} still hold. Yet
the representability of the functor of morphisms by a scheme 
is unclear: the Hilbert functor of a complete variety is 
generally not represented by a scheme (see e.g. 
\cite[Ex.~5.5.1]{Kollar96}), but no such example seems to be
known for morphisms.

\section{Proofs of Theorem \ref{thm:aut}
and of Corollary \ref{cor:der}}
\label{sec:aut}

\subsection{Proof of Theorem \ref{thm:aut} in characteristic $0$}
\label{subsec:zero}
\smartqed

We begin by setting notation and recalling a standard result 
of Galois descent, for any perfect field $k$.

We fix an algebraic closure $\bar{k}$ of $k$, and denote by
$\Gamma$ the Galois group of $\bar{k}/k$. For any scheme $X$,
we denote by $X_{\bar{k}}$ the $\bar{k}$-scheme obtained from $X$ 
by the base change $\Spec(\bar{k}) \to \Spec(k)$. Then $X_{\bar{k}}$
is equipped with an action of $\Gamma$ such that the structure
map $X_{\bar{k}} \to \Spec(\bar{k})$ is equivariant; moreover,
the natural morphism $X_{\bar{k}} \to X$ may be viewed as 
the quotient by this action. The assignement 
$Y \mapsto Y_{\bar{k}}$ defines a bijective correspondence 
between closed subschemes of $X$ and $\Gamma$-stable closed 
subschemes of $X_{\bar{k}}$.

Next, recall Chevalley's structure theorem: every connected
algebraic group $G$ has a largest closed connected linear normal 
subgroup $L$, and the quotient $G/L$ is an abelian variety 
(see e.g. \cite[Thm.~1.1.1]{BSU} when $k = \bar{k}$; the general 
case follows by the above result of Galois descent). 

We shall also need the existence of a normal projective
equivariant compactification of $G$, in the following sense:

\begin{lemma}\label{lem:comp}
There exists a normal projective variety $Y$ equipped with an action 
of $G \times G$ and containing an open orbit isomorphic to $G$,
where $G \times G$ acts on $G$ by left and right multiplication. 
\end{lemma}

\begin{proof}
\smartqed
When $k$ is algebraically closed, this statement is 
\cite[Prop.~3.1.1 (iv)]{BSU}. For an arbitrary $k$, we adapt 
the argument of [loc.~cit.]. 

If $G = L$ is linear, then we may identify it to a closed subgroup 
of some $\GL_n$, and hence of $\PGL_{n+1}$. The latter group has an 
equivariant compactification by the projectivization, $\bP(M_{n+1})$, 
of the space of matrices of size $(n+1) \times (n+1)$, where 
$\PGL_{n+1} \times \PGL_{n+1}$ acts via the action of 
$\GL_{n+1} \times \GL_{n+1}$ on $M_{n+1}$ by left and right matrix
multiplication. Thus, we may take for $Y$ the normalization of 
the closure of $L$ in $\bP(M_{n+1})$.

In the general case, choose a normal projective equivariant
compactification $Z$ of $L$ and let 
\[ Y := G \times^L Z \longrightarrow G/L \] 
be the fiber bundle associated with the principal $L$-bundle
$G \to G/L$ and with the $L$-variety $Z$, where $L$ acts 
on the left. Then $Y$ is a normal projective variety, since
so is $L$ and hence $L$ has an ample $L \times L$-linearized
line bundle. Moreover, $Y$ is equipped with a $G$-action
having an open orbit, $G \times^L L \cong G$. 

We now extend this $G$-action to an action of $G \times G$,
where the open $G$-orbit is identified to the 
$G \times G$-homogeneous space $(G \times G)/\diag(G)$,
and the original $G$-action, to the action of $G \times e_G$. 
For this, consider the scheme-theoretic center $Z(G)$ 
(resp.~$Z(L)$) of $G$ (resp.~$L$). Then $Z(L) = Z(G) \cap L$, 
since $Z(L)_{\bar{k}} = Z(G)_{\bar{k}} \cap L_{\bar{k}}$ in view 
of \cite[Prop.~3.1.1 (ii)]{BSU}. Moreover, 
$G_{\bar{k}} = Z(G)_{\bar{k}} L_{\bar{k}}$ by [loc.~cit.]; hence
the natural map $Z(G)/Z(L) \to G/L$ is an isomorphism of
group schemes. It follows that $G/L$ is isomorphic to 
\[ (Z(G) \times Z(G))/(Z(L) \times Z(L)) \diag(Z(G)) 
\cong (G \times G)/(L \times L) \diag(Z(G)). \]
Moreover, the $L \times L$-action on $Z$ extends to an action
of $(L \times L) \diag(Z(G))$, where $Z(G)$ acts trivially:
indeed, $(L \times L) \diag(Z(G))$ is isomorphic to 
\[ (L \times L \times Z(G))/(L \times L) \cap \diag(Z(G)) 
= (L \times L \times Z(G))/\diag(Z(L)), \]
and the subgroup scheme $\diag(Z(L)) \subset L \times L$ acts
trivially on $Z$ by construction. This yields an isomorphism
\[ G \times^L Z \cong 
(G \times G)\times^{(L \times L) \diag(Z(G))} Z, \]
which provides the desired action of $G \times G$.
\qed
\end{proof}

From now on in this subsection, we assume that $\charc(k) = 0$. 
We shall construct the desired variety $X$ from the equivariant 
compactification $Y$, by proving a succession of lemmas.

Denote by $\Aut^G(Y)$ the subgroup scheme of $\Aut(Y)$ consisting
of automorphisms which commute with the left $G$-action. Then 
the right $G$-action on $Y$ yields a homomorphism
\[ \varphi : G \longrightarrow \Aut^G(Y). \] 

\begin{lemma}\label{lem:autg}
With the above notation, $\varphi$ is an isomorphism.
\end{lemma}

\begin{proof}
\smartqed
Note that $\Aut^G(Y)$ stabilizes the open orbit for the left 
$G$-action, and this orbit is isomorphic to $G$. This yields 
a homomorphism $\Aut^G(Y) \to \Aut^G(G)$. Moreover, 
$\Aut^G(G) \cong G$ via the action of $G$ on itself by right 
multiplication, and the resulting homomorphism 
$\psi : \Aut^G(Y) \to G$ is readily seen to be inverse of 
$\varphi$.
\qed
\end{proof}

\begin{lemma}\label{lem:gen}
There exists a finite subset $F \subset G(\bar{k})$ 
which generates a dense subgroup of $G_{\bar{k}}$.
\end{lemma}

\begin{proof}
\smartqed
We may assume that $k = \bar{k}$.
If the statement holds for some closed normal subgroup $H$ of $G$ 
and for the quotient group $G/H$, then it clearly holds for $G$. 
Thus, we may assume that $G$ is simple, in the sense that it has 
no proper closed connected normal subgroup. Then, by Chevalley's
structure theorem, $G$ is either a linear algebraic group or an 
abelian variety. In the latter case, there exists $g \in G(k)$ 
of infinite order, and every such point generates a dense subgroup 
of $G$ (actually, every abelian variety, not necessarily simple, 
is topologically generated by some $k$-rational point, see 
\cite[Thm.~9]{FPS}). In the former case, $G$ is either the additive 
group $\bG_a$, the multiplicative group $\bG_m$, or a connected 
semisimple group. Therefore, $G$ is generated by finitely many 
copies of $\bG_a$ and $\bG_m$, each of which is topologically 
generated by some $k$-rational point (specifically, by any 
nonzero $t \in k$ for $\bG_a$, and by any $u \in k^*$ of 
infinite order for $\bG_m$).
\qed
\end{proof}

Choose $F \subset G(\bar{k})$ as in Lemma \ref{lem:gen}.
We may further assume that $F$ contains $\id_Y$ and is stable 
under the action of the Galois group $\Gamma$; then 
$F = E_{\bar{k}}$ for a unique finite reduced subscheme 
$E \subset G$. We have
\[ \Aut^F(Y_{\bar{k}}) = \Aut^{G(\bar{k})}(Y_{\bar{k}}) = 
\Aut^{G_{\bar{k}}}(Y_{\bar{k}}) \]
and the latter is isomorphic to $G_{\bar{k}}$ via $\varphi$,  
in view of Lemma \ref{lem:autg}. Thus, $\varphi$ yields
an isomorphism $G \cong \Aut^E(Y)$. 

Next, we identify $G$ with a subgroup of $\Aut(Y \times Y)$ 
via the closed embedding of group schemes
\[ \iota : \Aut(Y) \longrightarrow \Aut(Y \times Y),
\quad \varphi \longmapsto \varphi \times \varphi. \]
For any $f \in F$, let 
$\Gamma_f \subset Y_{\bar{k}} \times Y_{\bar{k}}$ be the graph 
of $f$; in particular, $\Gamma_{\id_Y}$ is the diagonal, 
$\diag(Y_{\bar{k}})$. Then there exists a unique closed reduced 
subscheme $Z \subset Y \times Y$ such that 
$Z_{\bar{k}} = \bigcup_{f \in F} \Gamma_f$.
We may now state the following key observation:

\begin{lemma}\label{lem:graph}
With the above notation, we have
\[ \iota(G) = \Aut^o(Y \times Y,Z), \]
where the right-hand side denotes the neutral component 
of the stabilizer of $Z$ in $\Aut(Y \times Y)$.
\end{lemma}

\begin{proof}
\smartqed
We may assume that $k = \bar{k}$, so that 
$Z = \bigcup_{f \in F} \Gamma_f$. Moreover, by connectedness,
$\Aut^o(Y \times Y,Z)$ is the neutral component of the 
intersection $\bigcap_{f \in F} \Aut(Y \times Y,\Gamma_f)$.
On the other hand, 
$\Aut^o(Y \times Y,Z) \subset \Aut^o(Y \times Y)$, 
and the latter is isomorphic to $\Aut^o(Y) \times \Aut^o(Y)$ 
via the natural homomorphism 
\[ \Aut^o(Y) \times \Aut^o(Y) \longrightarrow \Aut^o(Y \times Y), 
\quad (\varphi,\psi) \longmapsto \varphi \times \psi \]
(see \cite[Cor.~4.2.7]{BSU}).
Also, $\varphi \times \psi$ stabilizes a graph 
$\Gamma_f$ if and only if $\psi f = f \varphi$. In particular, 
$\varphi \times \psi$ stabilizes $\diag(Y) = \Gamma_{\id_Y}$ iff 
$\psi = \varphi$, and $\varphi \times \varphi$ stabilizes 
$\Gamma_f$ iff $\varphi$ commutes with $f$. As a consequence,
$\Aut^o(Y \times Y,Z)$ is the neutral component of 
$\iota(\Aut^F(Y))$. Since $\Aut^F(Y) = G$ is connected, this 
yields the assertion.
\qed
\end{proof}

Next, denote by $X$ the normalization of the blow-up 
of $Y \times Y$ along $Z$. Then $X$ is a normal projective 
variety equipped with a birational morphism
\[ \pi : X  \longrightarrow Y \times Y \]
which induces a homomorphism of group schemes
\[ \pi^* : G \longrightarrow \Aut(X), \]
since $Z$ is stable under the action of $G$ on $Y \times Y$.

\begin{lemma}\label{lem:blow}
Keep the above notation and assume that $n \geq 2$. Then 
$\pi^*$ yields an isomorphism of algebraic groups 
$G \stackrel{\cong}{\longrightarrow} \Aut^o(X)$.
\end{lemma}

\begin{proof}
\smartqed
It suffices to show the assertion after base change to 
$\Spec(\bar{k})$; thus, we may assume again that $k = \bar{k}$.

The morphism $\pi$ is proper and birational, and $Y \times Y$ is 
normal (since normality is preserved under separable field extension).
Thus, $\pi_*(\cO_X) = \cO_{Y \times Y}$ by Zariski's Main Theorem. 
It follows that $\pi$ induces a homomorphism of algebraic groups
\[ \pi_* : \Aut^o(X) \to \Aut^o(Y \times Y) \]
(see e.g. \cite[Cor.~4.2.6]{BSU}). In particular, $\Aut^o(X)$
preserves the fibers of $\pi$, and hence stabilizes 
the exceptional divisor $E$ of that morphism; as a consequence, 
the image of $\pi_*$ stabilizes $\pi(E)$. By connectedness, 
this image stabilizes every irreducible component of $\pi(E)$; 
but these components are exactly the graphs $\Gamma_f$, 
where $f \in F$ (since the codimension of any such graph 
in $Y \times Y$ is $\dim(Y) = n \geq 2$). 
Thus, the image of $\pi_*$ is contained in $\iota(G)$:
we may view $\pi_*$ as a homomorphism $\Aut^o(X) \to G$.
Since $\pi$ is birational, both maps $\pi^*$, $\pi_*$ 
are injective and the composition $\pi_* \pi^*$ is the identity. 
It follows that $\pi_*$ is inverse to $\pi^*$.
\qed
\end{proof}

We may now complete the proof of Theorem \ref{thm:aut} 
when $n \geq 2$. Let $X$ be as in Lemma \ref{lem:blow};
then $X$ admits an equivariant desingularization, i..e., 
there exists a smooth projective variety $X'$ equipped 
with an action of $G$ and with a $G$-equivariant birational 
morphism
\[ f : X' \longrightarrow X \]
(see \cite[Prop.~3.9.1, Thm.~3.36]{Kollar07}). 
We check that the resulting homomorphism of algebraic groups
\[ f^* : G \longrightarrow \Aut^o(X') \] 
is an isomorphism. For this, we may assume that 
$k = \bar{k}$; then again, \cite[Cor.~4.2.6]{BSU} 
yields a homomorphism of algebraic groups 
\[ f_* : \Aut^o(X') \longrightarrow \Aut^o(X) = G \]
which is easily seen to be inverse of $f^*$.

On the other hand, if $n = 1$, then $G$ is either 
an elliptic curve, or $\bG_a$, or a $k$-form of $\bG_m$.
We now construct a smooth projective surface $X$ such that 
$\Aut^o(X) \cong G$, via case-by-case elementary arguments.

When $G$ is an elliptic curve, we have $G \cong \Aut^o(G)$ 
via the action of $G$ by translations on itself. It follows that 
$G \cong \Aut^o(G \times C)$, where $C$ is any smooth projective 
curve of genus $\geq 2$. 

When $G = \bG_a$, we view $G$ as the group of automorphisms 
of the projective line $\bP^1$ that fix the point $\infty$ 
and the tangent line at that point, $T_{\infty}(\bP^1)$. 
Choose $x \in \bP^1(k)$ such that $0$, $x$, $\infty$ are all
distinct, and let $X$ be the smooth projective surface obtained 
by blowing up $\bP^1 \times \bP^1$ at the three points 
$(\infty,0)$, $(\infty,x)$, and $(\infty,\infty)$. 
Arguing as in the proof of Lemma \ref{lem:blow}, one checks 
that $\Aut^o(X)$ is isomorphic to the neutral component of
the stabilizer of these three points, in 
$\Aut^o(\bP^1 \times \bP^1) \cong \PGL_2 \times \PGL_2$. 
This identifies $\Aut^o(X)$ with the stabilizer of $\infty$ 
in $\PGL_2$, i.e, with the automorphism group $\Aff_1$ of 
the affine line, acting on the first copy of $\bP^1$. 
Thus, $\Aut^o(X)$ acts on each exceptional line via the natural 
action of $\Aff_1$ on $\bP(T_{\infty}(\bP^1) \oplus k)$, 
with an obvious notation; this action factors through an action
of $\bG_m = \Aff_1/\bG_a$, isomorphic to the $\bG_m$-action
on $\bP^1$ by multiplication. Let $X'$ be the smooth 
projective surface obtained by blowing up $X$ at a $k$-rational 
point of some exceptional line, distinct from $0$ and $\infty$; 
then $\Aut^o(X') \cong \bG_a$.

Finally, when $G$ is a $k$-form of $\bG_m$, we consider the smooth 
projective curve $C$ that contains $G$ as a dense open subset;
then $C$ is a $k$-form of the projective line $\bP^1$ on which
$\bG_m$ acts by multiplication. Thus,
the complement $P := C \setminus G$ is a point of degree $2$ 
on $C$ (a $k$-form of $\{ 0, \infty \}$); 
moreover, $G$ is identified with the stabilizer of $P$ in 
$\Aut(C)$. Let $X$ be the smooth projective surface obtained 
by blowing up $C \times C$ at 
$(P \times P) \cup (P \times e_G)$, where the neutral element
$e_G$ is viewed as a $k$-point of $C$. Arguing as in the proof 
of Lemma \ref{lem:blow} again, one checks that $\Aut^o(X)$ 
is isomorphic to the neutral component of the stabilizer of 
$(P \times P) \cup (P \times e_G)$ in 
$\Aut^o(C \times C) \cong \Aut^0(C) \times \Aut^0(C)$, 
i.e., to $G$ acting on the first copy of $C$.
\qed

\begin{remark}\label{rem:complex}
One may ask for analogues of Theorem \ref{thm:aut} for automorphism
groups of compact complex spaces. Given any such space $X$, the group 
of biholomorphisms, $\Aut(X)$, has the structure of a complex Lie group
acting biholomorphically on $X$ (see \cite{Douady}). If $X$ is K\"ahler,
or more generally in Fujiki's class $\cC$, then the neutral component
$\Aut^0(X) =: G$ has a meromorphic structure, i.e., a compactification 
$G^*$ such that the multiplication $G \times G \to G$ extends to a 
meromorphic map $G^* \times G^* \to G^*$ which is holomorphic on 
$(G \times G^*) \cup (G^* \times G)$; moreover, $G$ is K\"ahler and
acts biholomorphically and meromorphically on $X$ (see 
\cite[Th.~5.5, Cor.~5.7]{Fujiki}).

Conversely, every connected meromorphic K\"ahler group of 
dimension $n$ is the connected automorphism group of 
some compact K\"ahler manifold of dimension $2n$; indeed, the 
above arguments adapt readily to that setting. But it seems to be 
unknown whether any connected complex Lie group can be realized as 
the connected automorphism group of some compact complex manifold.
\end{remark}

\subsection{Proof of Theorem \ref{thm:aut} in prime characteristic}
\label{subsec:prime}
\smartqed

In this subsection, the base field $k$ is assumed to be perfect,
of characteristic $p > 0$.
Let $Y$ be an equivariant compactification of $G$ as in 
Lemma \ref{lem:comp}. Consider the closed subgroup scheme 
$\Aut^G(Y) \subset \Aut(Y)$, defined as the centralizer of $G$ 
acting on the left; then the $G$-action on the right still yields 
a homomorphism of group schemes $\varphi : G \to \Aut^G(Y)$.

As in Lemma \ref{lem:autg}, \emph{$\varphi$ is an isomorphism}.
To check this claim, note that $\varphi$ induces an isomorphism 
$G(\bar{k}) \stackrel{\cong}{\longrightarrow} \Aut^G(Y)(\bar{k})$ 
by the argument of that lemma.
Moreover, the induced homomorphism of Lie algebras
\[ \Lie(\varphi): \Lie(G) \longrightarrow \Lie \, \Aut^G(Y) \]
is an isomorphism as well: indeed, $\Lie(\varphi)$
is identified with the natural map
\[ \psi : \Lie(G) \longrightarrow \Der^G(\cO_Y), \]
where $\Der^G(\cO_Y)$ denotes the Lie algebra of 
left $G$-invariant derivations of $\cO_Y$. Furthermore, the restriction 
to the open dense subset $G$ of $Y$ yields an injective map
\[ \eta : \Der^G(\cO_Y) \longrightarrow \Der^G(\cO_G) \cong \Lie(G) \]
such that $\eta \psi = \id$; thus, $\eta$ is the inverse of 
$\psi$. It follows that $\Aut^G(Y)$ is reduced; this completes
the proof of the claim.

Next, \emph{Lemma \ref{lem:gen} fails in positive characteristics},
already for $\bG_a$ since every finite subset of $\bar{k}$ 
generates a finite additive group; that lemma also fails for $\bG_m$ 
when $\bar{k}$ is the algebraic closure of a finite field. Yet we have 
the following replacement: 
 
\begin{lemma}\label{lem:noeth}
With the above notation, there exists a finite subset $F$ of 
$G(\bar{k})$ such that 
$\Aut^{G_{\bar{k}}}(Y_{\bar{k}}) = \Aut^{F,o}(Y_{\bar{k}})$,
where the right-hand side denotes the neutral component of 
$\Aut^F(Y_{\bar{k}})$.
\end{lemma}

\begin{proof}
\smartqed
We may assume that $k = \bar{k}$; then $\Aut^G(Y) = \Aut^{G(k)}(Y)$.
The subgroup schemes $\Aut^{E,o}(Y)$, where $E$ runs 
over the finite subsets of $G(k)$, form a family of closed 
subschemes of $\Aut^o(Y)$. Thus, there exists a minimal such 
subgroup scheme, say, $\Aut^{F,o}(Y)$. For any $g \in G(k)$, 
the subgroup scheme $\Aut^{F \cup \{g \}, o}(Y)$ is contained 
in $\Aut^{F,o}(Y)$; thus, equality holds by minimality. 
In other words, $\Aut^{F,o}(Y)$ centralizes $g$; hence 
$F$ satisfies the assertion.
\qed
\end{proof}

It follows from Lemmas \ref{lem:autg} and \ref{lem:noeth} that 
$\Aut^{F,o}(Y_{\bar{k}}) \cong G_{\bar{k}}$ for some finite subset 
$F \subset G(\bar{k})$; we may assume again that $F$ contains $\id_Y$ 
and is stable under the action of the Galois group $\Gamma$. 
Thus, $G \cong \Aut^E(Y)$, where $E \subset G$ denotes the finite
reduced subscheme such that $E_{\bar{k}} = F$.

Next, \emph{Lemma \ref{lem:graph} still holds} 
with the same proof, in view of \cite[Cor.~4.2.7]{BSU}. In other 
words, we may again identify $G$ with the connected stabilizer 
in $\Aut(Y \times Y)$ of the unique closed reduced subscheme 
$Z \subset Y \times Y$ such that 
$Z_{\bar{k}} = \bigcup_{f \in F} \Gamma_f$.

Consider again the morphism $\pi : X \to Y \times Y$ obtained as 
the normalization of the blow-up of $Z$. Then $X$ is a normal 
projective variety, and $\pi$ induces a homomorphism of group schemes
\[ \pi^* : G \longrightarrow \Aut^o(X). \]
Now the statement of Lemma \ref{lem:blow} adapts as follows:

\begin{lemma}\label{lem:red} 
Keep the above notation and assume that $n \geq 2$. Then $\pi^*$ 
yields an isomorphism of algebraic groups
$G \stackrel{\cong}{\longrightarrow} \Aut^o(X)_{\red}$.
\end{lemma}

\begin{proof}
\smartqed
Using the fact that normality is preserved under separable
field extension, we may assume that $k = \bar{k}$. By 
\cite[Cor.~4.2.6]{BSU} again, we have a homomorphism of group schemes
\[ \pi_* : \Aut^o(X) \longrightarrow \Aut^o(Y \times Y) \]
and hence a homomorphism of algebraic groups  
\[ \pi_{*,\red} : \Aut^o(X)_{\red} \longrightarrow 
\Aut^o(Y \times Y)_{\red}. \]
Arguing as in the proof of Lemma \ref{lem:blow}, one checks 
that $\pi_{*,\red}$ maps $\Aut^o(X)_{\red}$ onto $\iota(G)$, 
and is injective on $k$-rational points. Also, the homorphism of 
Lie algebras $\Lie(\pi_{*,\red})$ is injective, as it extends
to a homomorphism
\[ \Lie(\pi_*) : \Lie \, \Aut^o(X) = \Der(\cO_X) \longrightarrow 
\Der(\cO_{Y \times Y}) = \Lie \, \Aut^o(Y \times Y) \]
which is injective, since $\pi$ is birational. Thus, we obtain
an isomorphism
$\pi_{*,\red}: \Aut^o(X)_{\red} 
\stackrel{\cong}{\longrightarrow} \iota(G)$
which is the inverse of $\pi^*$.
\qed
\end{proof}

To realize $G$ as a connected automorphism group scheme,
we now prove:

\begin{lemma}\label{lem:lie}
With the above notation, the homomorphism of Lie algebras
\[ \Lie(\pi^*) : \Lie(G) \longrightarrow \Der(\cO_X) \]
is an isomorphism if $n \geq 2$ and $n-1$ is not a multiple of 
$p$. 
\end{lemma}

\begin{proof}
\smartqed
We may assume again that $k = \bar{k}$. 
Since $\pi$ is birational, both maps $\Lie(\pi^*)$ and 
$\Lie(\pi_*)$ are injective and the composition 
$\Lie(\pi_*) \, \Lie(\pi^*)$ is the identity of $\Lie(G)$. 
Thus, it suffices to show that the image of $\Lie(\pi_*)$ 
is contained in $\Lie(G)$. For this, we use the natural 
action of $\Der(\cO_X)$ on the jacobian
ideal of $\pi$, defined as follows. Consider the sheaf
$\Omega^1_X$ of K\"ahler differentials on $X$. Recall that
$\Omega^1_X \cong \cI_{\diag(X)}/\cI_{\diag(X)}^2$ with an
obvious notation; thus, $\Omega^1_X$ is equipped with 
an $\Aut(X)$-linearization (see \cite[Exp.~I, \S 6]{SGA3}
for background on linearized sheaves, also called equivariant).
Likewise, $\Omega^1_{Y \times Y}$ is equipped with an
$\Aut(Y \times Y)$-linearization, and hence with an
$\Aut^o(X)$-linearization via the homomorphism $\pi_*$. 
Moreover, the natural map 
$\pi^*(\Omega^1_{Y \times Y}) \to \Omega^1_X$ 
is a morphism of $\Aut^0(X)$-linearized sheaves, since it
arises from the inclusion
$\pi^{-1}(\cI_{\diag(Y \times Y)}) \subset \cI_{\diag(X)}$.
This yields a morphism of $\Aut^o(X)$-linearized sheaves
\[ \pi^*(\Omega_{Y \times Y}^{2n}) 
\longrightarrow \Omega_X^{2n}. \]
Since the composition
\[ \Omega_X^{2n} \times Hom(\Omega_X^{2n}, \cO_X) 
\longrightarrow \cO_X \]
is also a morphism of linearized sheaves, we obtain 
a morphism of linearized sheaves
\[ Hom(\Omega_X^{2n}, \pi^*(\Omega_{Y \times Y}^{2n}))
\longrightarrow \cO_X \]
with image the jacobian ideal $\cI_{\pi}$. Thus, $\cI_{\pi}$
is equipped with an $\Aut^o(X)$-linearization. 
In particular, for any open subset $U$ of $X$, the Lie 
algebra $\Der(\cO_X)$ acts on $\cO(U)$ by derivations 
that stabilize $\Gamma(U,\cI_{\pi})$.

We now take $U = \pi^{-1}(V)$, where $V$ denotes the open
subset of $Y \times Y$ consisting of those smooth points
that belong to at most one of the graphs $\Gamma_f$.
Then the restriction
\[ \pi_U : U \longrightarrow V \] 
is the blow-up of the smooth variety $V$ along a closed
subscheme $W$, the disjoint union of smooth subvarieties 
of codimension $n$. Thus, $\cI_{\pi_U} = \cO_U(- (n-1) E)$,
where $E$ denotes the exceptional divisor of $\pi_U$. 
Hence we obtain an injective map 
\[ \Der(\cO_X) = \Der(\cO_X, \cI_{\pi}) \longrightarrow
\Der(\cO_U, \cO_U(- (n-1)E)), \]
with an obvious notation. Since $n-1$ is not a multiple 
of $p$, we have
\[ \Der(\cO_U, \cO_U(- (n-1)E)) = \Der(\cO_U, \cO_U(-E)). \]
(Indeed, if $D \in \Der(\cO_U, \cO_U(- (n-1)E))$ and 
$z$ is a local generator of $\cO_U(-E)$ at $x \in X$, then
$z^{n-1} \cO_{X,x}$ contains $D(z^{n-1}) = (n-1) z^{n-2} D(z)$,
and hence $D(z) \in z \cO_{X,x}$). Also, the natural map
\[ \Der(\cO_U) \longrightarrow \Der(\pi_{U,*}(\cO_U)) = \Der(\cO_V) \]
is injective and sends $\Der(\cO_U,\cO_U(-E))$ to 
$\Der(\cO_V,\pi_{U,*}(\cO_U(-E))$. Moreover, $\pi_{U,*}(\cO_U(-E))$
is the ideal sheaf of $W$, and hence is stable under $\Der(\cO_X)$
acting via the composition 
\[ \Der(\cO_X) \longrightarrow \Der(\pi_*(\cO_X)) = 
\Der(\cO_{Y \times Y}) \longrightarrow \Der(\cO_V). \] 
It follows that the image of $\Lie(\pi_*)$
stabilizes the ideal sheaf of the closure of $W$ in $Y \times Y$,
i.e., of the union of the graphs $\Gamma_f$. In view of Lemma 
\ref{lem:graph}, we conclude that $\Lie(\pi_*)$ sends
$\Der(\cO_X)$ to $\Lie(G)$.
\qed
\end{proof}

Lemmas \ref{lem:red} and \ref{lem:lie} yield an isomorphism 
$G \cong \Aut^o(X)$ when $n \geq 2$ and $p$ does not
divide $n - 1$. Next, when $n \geq 2$ and $p$ divides
$n - 1$, we choose a smooth projective curve $C$ of
genus $g \geq 2$, and consider $Y' := Y \times C$.
This is a normal projective variety
of dimension $n + 1$, equipped with an action of 
$G \times G$. Moreover, we have isomorphisms
\[ \Aut^o(Y) \stackrel{\cong}{\longrightarrow} 
\Aut^o(Y) \times \Aut^o(C)
\stackrel{\cong}{\longrightarrow} \Aut^o(Y'),
\quad \varphi \longmapsto \varphi \times \id_C \]
(where the second isomorphism follows again from
\cite[Cor.~4.2.6]{BSU}); this identifies 
$G \cong \Aut^{G,o}(Y)$ with $\Aut^{G,o}(Y')$.
We may thus replace everywhere $Y$ with $Y'$ 
in the above arguments, to obtain a normal projective 
variety $X'$ of dimension $2n + 2$ such that 
$\Aut^0(X') \cong G$.

Finally, if $n = 1$ then $G$ is again an elliptic curve, 
or $\bG_a$, or a $k$-form of $\bG_m$ (since every form of 
$\bG_a$ over a perfect field is trivial). It follows that 
$G \cong \Aut^o(X)$ for some smooth projective surface $X$, 
constructed as at the end of Subsection \ref{subsec:zero}.
\qed

\begin{remark}\label{rem:rat}
If $G$ is linear, then there exists a normal projective 
\emph{unirational} variety $X$ such that 
$\Aut^o(X)_{\red} \cong G$ and $\dim(X) = 2n$. Indeed,
$G$ itself is unirational (see \cite[Exp.~XIV, Cor.~6.10]{SGA3}, 
and hence so is the variety $X$ considered in the above proof 
when $n \geq 2$; on the other hand, when $n = 1$, the above proof 
yields a smooth projective rational surface $X$ such that 
$\Aut^o(X) \cong G$. If in addition $k$ is algebraically closed,
then $G$ is rational; hence we may further choose $X$ rational.

Conversely, if $X$ is a normal projective variety 
having a trivial Albanese variety (e.g., $X$ is unirational), 
then $\Aut^o(X)$ is linear. Indeed, the Albanese variety 
of $\Aut^o(X)_{\red}$ is trivial in view of 
\cite[Thm.~2]{Brion10}. Thus, $\Aut^o(X)_{\red}$ is affine
by Chevalley's structure theorem. It follows that $\Aut^o(X)$
is affine, or equivalently linear.

Returning to a connected linear algebraic group $G$,
the above proof adapts to show that there exists 
a normal projective unirational variety $X$ such that
$\Aut^o(X) \cong G$: in the argument after Lemma \ref{lem:lie},
it suffices to replace the curve $C$ with a normal projective 
rational variety $Z$ such that $\Aut^o(Z)$ is trivial and 
$\dim(Z) \geq 2$ is not a multiple of $p$. 
Such a variety may be obtained by blowing up $\bP^2$ at 
$4$ points in general position when $p \geq 3$; if $p = 2$, 
then we blow up $\bP^3$ along a smooth curve which is neither 
rational, nor contained in a plane.
\end{remark}

\begin{remark}\label{rem:nonper}
It is tempting to generalize the above proof to the setting of 
an arbitrary base field $k$. Yet this raises many technical 
difficulties; for instance, Chevalley's structure theorem 
fails over any imperfect field (see 
\cite[Exp.~XVII, App.~III, Prop.~5.1]{SGA3}, and \cite{Totaro} 
for a remedy). Also, normal varieties need not be geometrically 
normal, and hence the differential argument of Lemma 
\ref{lem:lie} also fails in that setting.
\end{remark}

\subsection{Proof of Corollary \ref{cor:der}}
\label{subsec:der}
\smartqed
(i)$\Rightarrow$(ii) Let $G := \Aut^o(X)$. Recall from 
\cite[Lem.~3.4, Thm.~3.7]{Matsumura-Oort} that
$G$ is a connected algebraic group with Lie algebra $\fg$. Also,
recall that 
\[ G_{\bar{k}} = Z(G)_{\bar{k}} \; L_{\bar{k}}, \] 
where $Z(G)$ denotes the center of $G$, and $L$ the largest 
closed connected normal linear subgroup of $G$. As a consequence, 
$\fg_{\bar{k}} = \Lie(Z(G))_{\bar{k}} + \Lie(L)_{\bar{k}}$. 
It follows that $\fg = \Lie(Z(G)) + \Lie(L)$, and hence 
we may choose a subspace $V \subset \Lie(Z(G))$ such that 
\[ \fg = V \oplus \Lie(L) \]
as vector spaces. This decomposition also holds as Lie algebras,
since $[V, V] = 0 = [V, \Lie(L)]$. Hence $\fg = \Lie(U \times L)$,
where $U$ is the (commutative, connected) unipotent algebraic group 
with Lie algebra $V$.

(ii)$\Rightarrow$(i) Let $G$ be a connected linear algebraic group
such that $\fg = \Lie(G)$. By Theorem \ref{thm:aut} and 
Remark \ref{rem:rat}, there exists a smooth projective unirational
variety $X$ of dimension $2n$ such that $G \cong \Aut^o(X)$; when
$k$ is algebraically closed, we may further choose $X$ rational.
Then of course $\fg \cong \Der(\cO_X)$. 
\qed

\section{Proofs of the statements about endomorphisms}
\label{sec:end}

\subsection{Proof of Proposition \ref{prop:end}}
\label{subsec:end}
\smartqed

(i) Since $C$ is connected and has a $k$-rational point, 
it is geometrically connected in view of 
\cite[Exp.~VIB, Lem.~2.1.2]{SGA3}. Likewise, the connected 
component of $i$ in $\Hom(Y,X)$ is geometrically connected.
To show the first assertion, we may thus assume that 
$k$ is algebraically closed. But then that assertion follows 
from \cite[Prop.~4.4.2, Rem.~4.4.3]{Brion12}.

The scheme-theoretic image of $C \times C$ under 
the morphism 
\[ \End(X) \times \End(X) \longrightarrow \End(X), \quad
(f,g) \longmapsto gf \]
is connected and contains $e^2 = e$; thus, this image 
is contained in $C$. Therefore, $C$ is a subsemigroup scheme 
of $\End(X)$. Also, every $g \in \Hom(Y,X)$ satisfies 
$g r e = g r i r = g r$. Thus, $f = f e$ for any $f \in C$.

(ii) Since $(i f_1 r) (i f_2 r) = i f_1 f_2 r$ for all
$f_1, f_2 \in \End(Y)$, we see that $\lambda_i \rho_r$ is 
a homomorphism of semigroup schemes which sends $\id_Y$ to
$e$. Also, $e i f r = i r i f r = i f r$ for all 
$f \in \End(Y)$, so that $\lambda_i \rho_r$ sends $\End(Y)$
to $e \End(X)$. Since $Y$ is a projective variety,
$\Aut^o(Y)$ is the connected component of $i$ in $\End(Y)$,
and hence is sent by $\lambda_i \rho_r$ to 
$C \cap e \End(X) = e C$. 

To show that $\lambda_i \rho_r$ is an isomorphism, note that
$e C= e  C e = i r C i r$ by (i). Moreover, the morphism
\[ \lambda_r \rho_i : \End(X) \longrightarrow \End(Y),
\quad  f \longmapsto r f i \]
sends $e$ to $\id_Y$, and hence $C$ to $\Aut^o(Y)$. 
Finally, 
$\lambda_r \rho_i (\lambda_i \rho_r(f)) = r (i f r) i = f$
for all $f \in \End(Y)$, and 
$\lambda_i \rho_r (\lambda_r \rho_i(f)) = i (r f i ) r
= e f e = f$ for all $f \in eC$. Thus, $\lambda_r \rho_i$
is the desired inverse.

(iii) Let $f \in C$ such that $f^2 = f$. Then $f e f = f$ by (i),
and hence $e f \in eC$ is idempotent. But $eC$ is a group scheme
by (ii); thus, $e f = e$. Write $f = g r$, where $g$ is a point 
of the connected component of $i$ in $\Hom(Y,X)$. Then
$e g r  = e$ and hence $r g r = r$, so that $r g = \id_Y$.
Conversely, if $g \in \Sec(r)$, then $g r$ is idempotent
as already noted. This shows the first assertion. For the second
assertion, just note that $(g_1r)(g_2r) = g_1 r$ for all
$g_1, g_2 \in \Sec(r)$. 

(iv) We have with an obvious notation 
$\varphi(f_1,g_1) \varphi(f_2,g_2) = f_1 g_1 f_2 g_2 
= f_1 g_1 e f_2 g_2$ by (i). Since $e f_2 = e$ by (iv),
it follows that 
$\varphi(f_1,g_1) \varphi(f_2,g_2) = f_1 g_1 e g_2 = f_1  g_1 g_2$.
Thus, $\varphi$ is a homomorphism of semigroup schemes.

We now construct the inverse of $\varphi$. Let $f \in C$; then
$e f \in e C$ has a unique inverse, $(ef)^{-1}$, in $eC$. Moreover,
$f = f e = f (ef)^{-1} ef$ and $f (ef)^{-1}$ is idempotent,
since 
\[ f (ef)^{-1} f (ef)^{-1} = f (ef)^{-1} ef (ef)^{-1} = 
f (ef)^{-1} e = f (ef)^{-1}. \]
We may thus define a morphism 
\[ \psi : C \longrightarrow E(C) \times eC, \quad
f \longmapsto (f (ef)^{-1}, ef). \] 
Then 
$\varphi\psi(f) = f (ef)^{-1} e f = f e = f$ for all $f \in C$, 
and 
$\psi \varphi(f,g) = (fg (efg)^{-1}, ef g) = (fg g^{-1}, eg)
= (f,g)$ for all $f \in E(C)$ and $g \in eC$. Thus, $\psi$ is 
the desired inverse.
\qed

\subsection{Proof of Proposition \ref{prop:semi}}
\label{subsec:semi}
\smartqed

(i) Consider the connected component $C$ of $\End(X)$ that
contains $S$. Then $C$ is of finite type, and hence so is $S$.
Choose a $k$-rational point $f$ of $S$ and denote by 
$\langle f \rangle$ the smallest closed subscheme of $S$ 
containing all the powers $f^n$, where $n \geq 1$. Then 
$\langle f \rangle$ is a reduced commutative subsemigroup 
scheme of $S$. By the main result of \cite{Brion-Renner}, 
it follows that $\langle f \rangle$ has an idempotent 
$k$-rational point. In particular, $E(S)$ has a $k$-rational 
point. 

Since $E(S) \subset E(C)$, we have $f_1 f_2 = f_1$ for 
any $f_1,f_2 \in E(S)$, by Proposition \ref{prop:end}.
It remains to show that $E(S)$ is connected; this will follow 
from (ii) in view of the connectedness of $S$.

(ii) By Proposition \ref{prop:end} again, $\varphi$
yields an isomorphism 
$E(C) \times eC \stackrel{\cong}{\longrightarrow} C$.
Moreover, $f e = f$ for all $f \in C$, and $eC = e C e$ is 
a group scheme. Thus, $e S = e S e$ is a submonoid scheme
of $eC$, and hence a closed subgroup scheme by Lemma 
\ref{lem:sub} below. In other words, $e f$ is invertible 
in $eS$ for any $f \in S$. One may now check as in the proof 
of Proposition \ref{prop:end} (iv) that the morphism
\[ \psi: S \longrightarrow E(S) \times eS, \quad 
f \longmapsto ( f (e f)^{-1}, e f) \] 
yields an isomorphism of semigroup schemes, with inverse 
$\varphi$.

(iii) Since $\varphi$ is an isomorphism of semigroup schemes,
$\pi$ is a homomorphism of such schemes. Moreover,
$\pi(f) = f (ef)^{-1}$ for all $f \in S$, since $\psi$ is 
the inverse of $\varphi$. If $f \in E(S)$, then 
$f = f^2 = f e f$ and hence $\pi(f) = f ef (ef)^{-1} = f e = f$.
Thus, $\pi$ is a retraction.

Let $\rho: S \to E(S)$ be a retraction of semigroup schemes.
For any $f \in S$, we have
$\rho(f) = \rho(f (ef)^{-1} ef) = \rho(f(ef)^{-1}) \rho(ef)$.
Moreover, $\rho(f(ef)^{-1})= f(ef)^{-1}$, since 
$f(ef)^{-1} \in E(S)$; also, 
$\rho(ef) = \rho(ef) \rho((ef)^{-1}) = \rho(e) = e$. Hence
$\rho(f) = f (ef)^{-1} e = f (ef)^{-1} = \pi(f)$.
\qed

\begin{lemma}\label{lem:sub}
Let $G$ be a group scheme of finite type, and $S \subset G$ 
a subsemigroup scheme. Then $S$ is a closed subgroup scheme.
\end{lemma}

\begin{proof}
\smartqed
We have to prove that $S$ is closed, and stable under the 
automorphism $g \mapsto g^{-1}$ of $G$. It suffices to check 
these assertions after base extension to any larger field;
hence we may assume that $k$ is algebraically closed.

Arguing as at the beginnning of the proof of Proposition 
\ref{prop:semi} (i), we see that $S$ has an idempotent
$k$-rational point; hence $S$ contains the neutral element,
$e_G$. In other words, $S$ is a submonoid scheme of $G$.
By Lemma \ref{lem:unit} below, there exists an open subgroup
scheme $G(S) \subset S$ which represents the invertibles in 
$S$. In particular, $G(S)_{\red}$ is the unit group of the
algebraic monoid $S_{\red}$. Since that monoid has a unique 
idempotent, it is an algebraic group by 
\cite[Prop.~2.2.5]{Brion12}. In other words, we have 
$G(S)_{\red} = S_{\red}$. As $G(S)$ is open in $S$, it follows 
that $G(S) = S$. Thus, $S$ is a subgroup scheme of $G$, 
and hence is closed by \cite[Exp.~VIA, Cor.~0.5.2]{SGA3}. 
\qed
\end{proof}

To complete the proof, it remains to show the following
result of independent interest:

\begin{lemma}\label{lem:unit}
Let $M$ be a monoid scheme of finite type. Then the group 
functor of invertibles of $M$ is represented by a group 
scheme $G(M)$, open in $M$.
\end{lemma}

\begin{proof}
\smartqed
We first adapt the proof of the corresponding statement for
(reduced) algebraic monoids (see \cite[Thm.~2.2.4]{Brion12}).
Denote for simplicity the composition law of $M$ by
$(x,y) \mapsto xy$, and the neutral element by $1$.
Consider the closed subscheme $G \subset M \times M$ defined 
in set-theoretic notation by 
\[ G = \{ (x,y) \in M \times M ~\vert~ xy = yx = 1 \}. \]
Then $G$ is a subgroup scheme of the monoid scheme
$M \times M^{\op}$, where $M^{\op}$ denotes the opposite 
monoid, that is, the scheme $M$ equipped with the composition
law $(x,y) \mapsto yx$. Moreover, the first projection 
\[ p : G \longrightarrow M \] 
is a homomorphism of monoid schemes, which sends 
the $T$-valued points of $G$ isomorphically to the $T$-valued 
invertible points of $M$ for any scheme $T$. It follows that
the group scheme $G$ represents the group functor of 
invertibles in $M$.

To complete the proof, it suffices to check that $p$ 
is an open immersion; for this, we may again assume 
that $k$ is algebraically closed. Clearly, $p$ is 
universally injective; we now show that it is \'etale. 
Since that condition defines an open subscheme of $G$,
stable under the action of $G(k)$ by left multiplication,
we only need to check that $p$ is \'etale at the neutral
element $1$ of $G$. For this, the argument of [loc.~cit.] 
does not adapt readily, and we shall rather consider the formal 
completion of $M$ at $1$,
\[ N := \Spf(\widehat{\cO}_{M,1}). \]
Then $N$ is a formal scheme having a unique point; moreover,
$N$ has a structure of formal monoid scheme, defined as follows.
The composition law $\mu : M \times M \to M$ sends $(1,1)$ to 
$1$, and hence yields a homomorphism of local rings
$\mu^{\#} : \cO_{M,1} \to \cO_{M \times M, (1,1)}$. 
In turn, $\mu^{\#}$ yields a homomorphism of completed local rings
\[ \Delta : \widehat{\cO}_{M,1} \longrightarrow 
\widehat{\cO}_{M \times M, (1,1)}
= \widehat{\cO}_{M,1} \, \widehat{\otimes} \, \widehat{\cO}_{M,1}. \]
We also have the homomorphism 
\[ \varepsilon : \widehat{\cO}_{M,1} \longrightarrow k \] 
associated with $1$. One readily checks that $\Delta$ and 
$\varepsilon$ satisfy conditions (i) (co-associativity) and (ii)
(co-unit) of \cite[Exp.~VIIB, 2.1]{SGA3}; hence 
they define a formal monoid scheme structure on $N$. In view 
of [loc.~cit., 2.7.~Prop.], it follows that $N$ is in fact
a group scheme. As a consequence, $p$ is an isomorphism after 
localization and completion at $1$; in other words, 
$p$ is \'etale at $1$.
\qed
\end{proof}

\begin{remark}\label{rem:semi}
Proposition \ref{prop:semi} gives back part of the description
of all algebraic semigroup structures on a projective variety
$X$, obtained in \cite[Thm.~4.3.1]{Brion12}. 

Specifically, every such structure $\mu : X \times X \to X$, 
$(x,y) \mapsto xy$ 
yields a homomorphism of semigroup schemes
$\lambda : X \longrightarrow \End(X)$, 
$x \longmapsto (y \mapsto xy)$
(the ``left regular representation''). Thus, $S := \lambda(X)$ 
is a closed subsemigroup scheme of $\End(X)$. Choose an idempotent
$e \in X(k)$. In view of Proposition \ref{prop:semi}, we have
$\lambda(x) \lambda(e) = \lambda(x)$ for all $x \in X$;
moreover, $\lambda(e) \lambda(x)$ is invertible in $\lambda(e)S$.
It follows that $x e y = xy$ for all $x,y \in X$. Moreover,
for any $x \in X$, there exists $y \in e X$ such that 
$y e x z = e x y z = e z$ for all $z \in X$. In particular,
$(exe) (eye) = (eye) (exe) = e$, and hence $eXe$ is an
algebraic group. 

These results are the main ingredients in the proof of 
\cite[Thm.~4.3.1]{Brion12}. They are deduced
there from the classical rigidity lemma, while the proof
of Proposition \ref{prop:semi} relies on a generalization
of that lemma. 
\end{remark}

\begin{remark}\label{rem:nopoint}
If $k$ is not algebraically closed, then connected semigroup 
schemes of endomorphisms may well have no $k$-rational point.
For example, let $X$ be a projective variety having no $k$-rational
point; then the subsemigroup scheme $S \subset \End(X)$ consisting 
of constant endomorphisms (i.e., of those endomorphisms that factor
through the inclusion of a closed point in $X$) is isomorphic to 
$X$ itself, equipped with the composition law $(x,y) \mapsto y$. 
Thus, $S$ has no $k$-rational point either.

Yet Proposition \ref{prop:semi} can be extended to any
geometrically connected subsemigroup scheme $S \subset \End(X)$, 
not necessarily having a $k$-rational point. Specifically, $E(S)$
is a nonempty, geometrically connected subsemigroup scheme, with
semigroup law given by $f_1 f_2 = f_1$. Moreover, there exists
a unique retraction of semigroup schemes 
\[ \pi : S \longrightarrow E(S); \]
it assigns to any point $f \in S$, the unique idempotent 
$e \in E(S)$ such that $e f = f$. Finally, the above morphism
$\pi$ defines a structure of $E(S)$-monoid scheme on $S$, with
composition law induced by that of $S$, and with neutral section
the inclusion 
\[ \iota : E(S) \longrightarrow S. \] 

In fact, this monoid scheme is a group scheme: consider indeed 
the closed subscheme $T \subset S \times S$ defined in 
set-theoretic notation by
\[ T = \{ (x,y) \in S \times S ~\vert~ 
xy = yx, \; x^2 y = x, \; xy^2 = y \}, \]
and the morphism 
\[ \rho: T \longrightarrow S, \quad (x,y) \longmapsto xy. \]
Then one may check that $\rho$ is a retraction from $T$ to
$E(S)$, with section 
\[ \varepsilon : E(S) \longrightarrow T, \quad 
x \longmapsto (x,x). \]
Moreover, $T$ is a group scheme over $E(S)$ via $\rho$, 
with composition law given by $(x,y) (x',y') := (xx',y'y)$,
neutral section $\varepsilon$, and inverse given by 
$(x,y)^{-1} := (y,x)$. Also, the first projection 
\[ p_1 : T \longrightarrow S, \quad (x,y) \longrightarrow x \]
is an isomorphism which identifies $\rho$ with $\pi$;
furthermore, $p_1$ is an isomorphism of monoid schemes. 
This yields the desired group scheme structure. 

When $k$ is algebraically closed, all these assertions are
easily deduced from the structure of $S$ obtained in Proposition 
\ref{prop:semi}; the case of an arbitrary field follows by descent.
\end{remark}

\subsection{Proof of Proposition \ref{prop:bound}}
\label{subsec:bound}

(i) can be deduced from the results of \cite{Brion-Renner}; 
we provide a self-contained proof by adapting some of 
the arguments from [loc.~cit.].

As in the beginning of the proof of Proposition 
\ref{prop:semi}, we denote by $\langle f \rangle$ the 
smallest closed subscheme of $\End(X)$ containing all 
the powers $f^n$, where $n \geq 1$. In view of the 
boundedness assumption, $\langle f \rangle$ is an algebraic
subsemigroup; clearly, it is also commutative. 
The subsemigroups $\langle f^m \rangle$, where $m \geq 1$, 
form a family of closed subschemes 
of $\langle f \rangle$; hence there exists a minimal 
such subsemigroup, $\langle f^{n_0} \rangle$. Since 
$\langle f^m \rangle \cap \langle f^n \rangle
\supset \langle f^{mn} \rangle$, we see that
$\langle f^{n_0} \rangle$ is the smallest such
subsemigroup. 

The connected components of $\langle f^{n_0} \rangle$ form 
a finite set $F$, equipped with a semigroup structure such 
that the natural map $\varphi : \langle f^{n_0} \rangle \to F$ 
is a homomorphism of semigroups. In particular, the finite
semigroup $F$ is generated by $\varphi(f^{n_0})$.
It follows readily that $F$ has a unique idempotent,
say $\varphi(f^{n_0n})$. Then the fiber 
$\varphi^{-1} \varphi(f^{n_0n})$
is a closed connected subsemigroup of 
$\langle f^{n_0} \rangle$, and contains 
$\langle f^{n_0 n} \rangle$. By the minimality 
assumption, we must have 
$\langle f^{n_0 n} \rangle = \varphi^{-1} \varphi(f^{n_0n})
= \langle f^{n_0} \rangle$. As a consequence,
$\langle f^{n_0} \rangle$ is connected.

Also, recall that $\langle f^{n_0} \rangle$ 
is commutative. In view of Proposition \ref{prop:semi}, 
it follows that this algebraic semigroup is in fact 
a group. In particular, $\langle f^{n_0} \rangle$ contains 
a unique idempotent, say $e$. Therefore, $e$ is also
the unique idempotent of $\langle f \rangle$:
indeed, if $g \in \langle f \rangle$ is idempotent, 
then $g = g^{n_0} \in \langle f^{n_0} \rangle$, 
and hence $g = e$.

Thus, $e \langle f \rangle = \langle e f \rangle$ 
is a closed submonoid of $\langle f \rangle$ 
with neutral element $e$ and no other idempotent. 
In view of \cite[Prop.~2.2.5]{Brion12}, it follows
that $e \langle f \rangle$ is a group, say, $G$.
Moreover, $f^{n_0} = e f^{n_0} \in e \langle f \rangle$, 
and hence $f^n \in G$ for all $n \geq n_0$.     
On the other hand, if $H$ is a closed subgroup
of $\End(X)$ and $n_1$ is a positive integer such that
$f^n \in H$ for all $n \geq n_1$, then $H$ contains 
$\langle f^{n_1} \rangle$, and hence 
$\langle f^{n_0} \rangle$ by minimality. In particular, 
the neutral element of $H$ is $e$. Let $g$ denote 
the inverse of $f^{n_1}$ in $H$; then $H$ contains
$f^{n_1 + 1} g = e f$, and hence $G \subset H$. Thus, 
$G$ satisfies the assertion.

(ii) Assume that $f^n(x) = x$ for some $n \geq 1$ and 
some $x \in X(k)$. Replacing $n$ with a large multiple,
we may assume that $f^n \in G$. Let $Y := e(X)$,
where $e$ is the neutral element of $G$ as above, 
and let $y := e(x)$. Then $Y$ is a closed 
subvariety of $X$, stable by $f$ and hence by $G$;
moreover, $G$ acts on $Y$ by automorphisms. Also,
$y \in Y$ is fixed by $f^n$. Since $f^n = (e f)^n$,
it follows that the $(ef)^m(y)$, where $m \geq 1$, 
form a finite set. As the positive powers of $e f$ are 
dense in $G$, the $G$-orbit of $y$ must be finite. Thus, 
$y$ is fixed by the neutral component $G^o$. In view 
of \cite[Prop.~2.1.6]{BSU}, it follows that $G^o$ 
is linear; hence so is $G$. 

Conversely, if $G$ is linear, then $G^o$ is a connected 
linear commutative algebraic group, and hence 
fixes some point $y \in Y(k)$ by Borel's fixed point 
theorem. Then $y$ is periodic for $f$.

Next, assume that $X$ is normal; then so is $Y$ 
by Lemma \ref{lem:nor} below. In view of 
\cite[Thm.~2]{Brion10}, it follows that
$G^o$ acts on the Albanese variety $A(Y)$ via a 
finite quotient of its own Albanese variety, $A(G^o)$. 
In particular, $G$ is linear if and only if $G^o$ 
acts trivially on $A(Y)$. Also, note that $A(Y)$ 
is isomorphic to a summand of the abelian 
variety $A(X)$: the image of the idempotent $A(e)$ 
induced by $e$. If $G^o$ acts trivially on $A(Y)$, 
then it acts on $A(X)$ via $A(e)$, since $G^o = G^o e$.
Thus, some positive power $f^n$ acts on $A(X)$ via
$A(e)$ as well. Conversely, if $f^n$ acts on $A(X)$ 
via some idempotent $g$, then we may assume that 
$f^n \in G^o$ by taking a large power. Thus, 
$f^n = e f^n = f^n e$ and hence $g = A(e) g = g A(e)$; 
in other words, $g$ acts on $A(X)$ as an idempotent 
of the summand $A(Y)$. On the other hand, $g = A(f^n)$ 
yields an automorphism of $A(Y)$; it follows that 
$g = A(e)$.

\begin{lemma}\label{lem:nor}
Let $X$ be a normal variety, and $r: X \to Y$ a retraction.
Then $Y$ is a normal variety as well.
\end{lemma}

\begin{proof}
\smartqed
Consider the normalization map, $\nu : \tilde{Y} \to Y$. 
By the universal property of $\nu$, there exists 
a unique morphism $\tilde{r}: X \to \tilde{Y}$ such that 
$r = \nu \tilde{r}$. Since $r$ has a section, so has $\nu$. 
As $\tilde{Y}$ is a variety and $\nu$ is finite, it follows 
that $\nu$ is an isomorphism.
\qed
\end{proof}

\begin{remark}\label{rem:mono}
With the notation of the proof of (i), the group $G$ is 
the closure of the subgroup generated by $e f$. Hence $G$ 
is \emph{monothetic} in the sense of \cite{FPS}, which 
obtains a complete description of this class of algebraic
groups. Examples of monothetic algebraic groups include all 
the semiabelian varieties, except when $k$ is the algebraic
closure of a finite field (then the monothetic algebraic
groups are exactly the finite cyclic groups).  
\end{remark}

\medskip

\noindent
{\bf Acknowledgements.} This work was began during a staying 
at Tsinghua University in March 2013. 
I warmly thank the Mathematical Sciences Center for its 
hospitality, and J\'er\'emy Blanc, Corrado De Concini,
St\'ephane Druel, Baohua Fu, Eduard Looijenga, J\"org Winkelmann
and De-Qi Zhang for helpful discussions or e-mail exchanges.

\end{document}